\documentclass{article}
\usepackage[T1]{fontenc}
\usepackage[utf8]{inputenc}
\usepackage{etex}

\usepackage{pgfplots}% This uses tikz
\pgfplotsset{compat=newest}% use newest version
\usepgfplotslibrary{external}
\usepackage[height=8.2in,width=6in]{geometry}
\usepackage{amsmath,amsthm,amssymb,bm}
\usepackage[sort&compress,numbers]{natbib}
\usepackage{graphicx,mathrsfs}
\usepackage{paralist}
\usepackage{enumerate}
\usepackage{stmaryrd}
\usepackage{fancyhdr}
\usepackage{colortbl}
\usepackage{xcolor}
\usepackage{microtype}
\usepackage{dsfont}
\usepackage{mathdots}
\usepackage{float}
\usepackage[colorlinks=true,citecolor=blue]{hyperref}

\newtheorem{theorem}{Theorem}[section]

\newtheorem{prop}[theorem]{Proposition}

\newtheorem{rmk}[theorem]{Remark}

\newcommand{\reals}{\mathbb{R}}
\newcommand{\half}{\tfrac{1}{2}}

\newcommand{\ind}{\mathds{1}}

\newcommand{\nlsum}{\sum\nolimits}

\newcommand{\norm}[1]{\|{#1}\|}

\DeclareMathOperator{\trace}{tr}
\DeclareMathOperator{\Diag}{Diag}

\numberwithin{equation}{section}

\begin{document} 
\title{Explicit diagonalization of a Cesar\'o matrix}
\date{Original: 9 Mar 2014\\ This revision: 4 Apr 2015}
\author{Suvrit Sra\thanks{This note was written when the author was on leave from Max Planck Institute for Intelligent Systems, T\"ubingen, Germany at Carnegie Mellon University, Pittsburgh, USA}.}

\maketitle 

\begin{abstract}
  We study a specific ``anti-triangular'' Cesar\'o matrix corresponding to a Markov chain. We derive closed forms for all the eigenvalues and eigenvectors of this matrix.
To complement our main result, we include a short section on the closely related kernel matrices $[\min(i,j)]$ and $[1/\max(i,j)]$.
\end{abstract}

\section{Introduction}
Eigenvalues of Markov chains lend insight into the speed of convergence to an invariant measure or stationary distribution. The corresponding eigenvector provides the distribution of the stationary state. In this paper, we study a particular Markov transition matrix of Cesar\'o type, and determine its eigenvalues and eigenvectors in closed form. This explicit determination of the spectrum may be of interest because our matrix is ``anti-triangular'', a class of matrices for which eigenvalue problems are considerably harder than for the usual triangular matrices. 

The matrix that we study is a Ces\'aro-like matrix, a name inspired by the classic article~\citep{choi} (see also the remark below). The specific matrix that we study is anti-lower triangular; such matrices have also been studied by~\citep{mehl}, and very recently by~\citet{ochSa14}, who undertake a detailed theoretical development. We also note in passing a potential connection to inverse eigenvalue problems for anti-bidiagonal matrices~\citep{holz}, and accurate numerical methods for anti-bidiagonal matrices \citep{koev07}, for which our explicit diagonalization can be used to provide a test case.

Without further ado, let us move onto the key problem in this paper: \emph{Diagonalize the following structured Markov transition matrix on $n$ states:}
\begin{equation}
  \label{eq:4}
  P = \begin{pmatrix}
    &&&&1\\
    &&&\half&\half\\
    &&\iddots&&\vdots\\
    \tfrac{1}{n} && \cdots && \tfrac{1}{n}
  \end{pmatrix}.
\end{equation}
Matrix~\eqref{eq:4} also appears in \citep[(2.3)]{ochSa14}, who subsequently obtain the eigenvalues of $P$ in closed form using their general theory. We obtain the eigenvalues explicitly using completely elementary means; in addition, we obtain a basis of eigenvectors that diagonalize $P$. Thus, we derive explicit formulae for the entries of a matrix $S$ such that  that $P=S\Lambda S^{-1}$, where $\Lambda$ is the diagonal matrix of eigenvalues. How do we know that $P$ is diagonalizable? Proposition~\ref{prop:eig} shows that $P$ has $n$ distinct eigenvalues which suffices for diagonalizability.

Before we begin, we remark that (unsurprisingly) diagonalizing the usual ``Cesar\'o matrix''~\citep{choi}
\begin{equation*}
  C =
  \begin{pmatrix}
    1&&&&\\
    \half&\half&&&\\
    &&\ddots&&\vdots\\
    \tfrac1n && \cdots &&\tfrac1n
  \end{pmatrix},
\end{equation*}
turns out to be much easier. It can be easily verified that $VCV^{-1}=\Diag([1/i]_{i=1}^n)$ for $V$ given by~\eqref{eq:9}; the anti-triangular form of~\eqref{eq:4} makes it harder to diagonalize explicitly.

\section{Explicit Diagonalization}
Let $J$ denote the \emph{reverse identity matrix}, i.e., the matrix with ones on its anti-diagonal. Then $JP$ is upper-triangular, so that eigenvalues of $JP$ can be read off of the diagonal. These eigenvalues are $\lambda_i(JP)=1/i$. It turns out that the eigenvalues of $P$ are also $1/i$ multiplied with alternating signs. Let us prove this observation.

\begin{prop}
  \label{prop:eig}
  Let $P$ be given by~\eqref{eq:4}. Then, $\lambda_i(P)=(-1)^{i+1}/i$ for $i=1,\ldots,n$.
\end{prop}
\begin{proof}
  It proves more convenient to analyze $P^{-1}$. An easy verification shows that $P^{-1}$ is the anti-bidiagonal matrix
\begin{equation}
  \label{eq:5}
  P^{-1} =
  {\footnotesize\begin{pmatrix}
      &&&1-n & n\\
      &&2-n & n-1&\\
      &\iddots&\iddots&\\
      -1& 2 &&&\\
      1&&&&
    \end{pmatrix}}.
\end{equation}
  To obtain eigenvalues of $P^{-1}$ it suffices to find a triangular matrix similar to it. To that end, we use the following matrix (curiously, this matrix is the strict lower-triangular part of $JP^{-1}$):
  \begin{equation*}
    L =
    {\footnotesize\begin{pmatrix}
        0 &&&\\
        -1 & 0 &&\\
        0& -2 & 0 &&\\
        &&\ddots&\ddots\\
        \dots& && 1-n & 0
      \end{pmatrix}}.
  \end{equation*}
  Setting $V=\exp(L)$, we obtain
  \begin{equation}
    \label{eq:9}
    V = [V_{ij}] = \left[(-1)^{(i-j)}\binom{i-1}{j-1}\right]\quad \text{for}\quad i \ge j.
  \end{equation}
  Explicitly carrying out the multiplication, we obtain the triangular matrix (using $V^{-1}=\exp(-L)$)
  \begin{equation}
    \label{eq:6}
    VP^{-1}V^{-1} = {
      \footnotesize
      \begin{pmatrix}
        1 & * & \cdots &*\\
        & -2 & * &* \\
        &&\ddots&*\\
        &&&(-1)^{n+1}n
      \end{pmatrix},
    }
  \end{equation}
  where $*$ represents unspecified entries. From~\eqref{eq:6} it is clear that $\lambda_i(P^{-1})=(-1)^{i+1}i$.
\end{proof}

Obtaining eigenvectors is harder and requires more work. A quick computation shows that $P^{-2}$ is tridiagonal but asymmetric, which rules out an easy solution. However, $VP^{-2}V^{-1}$ turns out to be a highly structured bidiagonal matrix:
\begin{equation*}
  B := VP^{-2}V^{-1} = {
    \footnotesize
    \begin{pmatrix}
      1 & -2(n-1)\\
      & 4 & -3(n-2)\\
      &&9 & -4(n-3)&\\
      &&&\ddots & \ddots\\
      \\
      &&&&(n-1)^2 & -n(1)\\
      &&&&& n^2
    \end{pmatrix}.
  }
\end{equation*}
Suppose now that there is a matrix $S$ that diagonalizes $B$, that is $SBS^{-1}=\Lambda$ with $\Lambda=[i^2]_{i=1}^n$ diagonal. Then,
\begin{equation}
  \label{eq:10}
  S^{-1}\Lambda S = VP^{-2}V^{-1} \implies SVP^{-2}V^{-1}S^{-1} = \Lambda = \text{Diag}([i^2]_{i=1}^n),
\end{equation}
which shows that $SV$ diagonalizes $P^{-2}$, completing the answer.

\begin{theorem}
  \label{thm:diag}
  Let $(x)_k := x(x+1)\cdots(x+k-1)$ be the Pochammer symbol, and let $M$ be lower-triangular with nonzero entries in column $j$ given by
  \begin{equation}
    \label{eq:7}
    m_{kj} := \frac{(j+1)_{k-j}(n-k+1)_{k-j}}{(2j+1)_{k-j}(k-j)!},\qquad k \ge j\quad (1\le j \le n).
  \end{equation}
  Then, $S=M^T$ diagonalizes $B$.
\end{theorem}
\begin{proof}
  To find $S$ we need to solve the system of equations:
\begin{equation*}
  SB = \Lambda S\quad\Longleftrightarrow\quad B^TM=M\Lambda.
\end{equation*}
Consider the $j$th eigenvalue $\lambda_j=j^2$; denote the corresponding column of $M$ by $m$ and its $k$th entry by $m_k$. To obtain $m$ we must solve the linear system
\begin{equation*}
  \begin{split}
    B^Tm = j^2m,\qquad \implies & m_1 = j^2m_1\\
    -(k+1)(n-k)m_k + (k+1)^2m_{k+1} &= j^2m_{k+1},\quad 1 \le k \le n.
\end{split}
\end{equation*}
Since $B$ is bidiagonal, a brief reflection shows that $M$ is lower-triangular with $1$s on its diagonal---the 1s come from the equation corresponding to the index $k+1=j$. The subsequent entries of $m$ are nonzero.  Symbolic computation with a few different values of $j$ suggests the general solution (which can be formally proved using an easy but tedious induction): 
\begin{equation*}
  m_{k} = \frac{(j+1)_{k-j}(n-k+1)_{k-j}}{(2j+1)_{k-j}(k-j)!},\qquad k \ge j.
\end{equation*}
%One may easily also numerically verify the  above claim using \textsc{Matlab}.
\end{proof}

\begin{rmk}[Added Oct 31, 2014]
  Similarly we can diagonalize the transition matrix
  \begin{equation}
    \label{eq:11}
    Z = \begin{pmatrix}
    \tfrac{1}{n} & \tfrac{1}{n}& \cdots && \tfrac{1}{n}\\
    \tfrac{1}{n-1} &\tfrac{1}{n}& \cdots & \tfrac{1}{n-1}& 0\\
    &\vdots&\iddots&&\vdots\\
    \tfrac12 & \tfrac12\\
    1
  \end{pmatrix}.
\end{equation}
Clearly, from a diagonalization of $P^{-1}$ we can recover a diagonalization of $Z$ above. To see why, observe that $P^{-1}=(J^TZ^{-1}J)^T$ where $J$ is the ``reverse identity'' (anti-diagonal) matrix.
\end{rmk}

\section{Cesar\'o Kernels}
\label{sec.kernels}

In this section we mention kernel matrices closely related to the Cesrar\'o matrix discussed above.

\subsection{Brownian bridge kernel}
Interestingly, matrix $P$ is closely related to the \emph{Brownian bridge} kernel $K=[k_{ij}]=\min(i,j)$~\citep{hofman06}. Proposition~\ref{prop:brown} makes this connection precise and highlights a well-known property of this kernel.
\begin{prop}
  \label{prop:brown}
  The kernel matrix $K=PP^T$ is \emph{infinitely divisible} (i.e., $[k_{ij}^r] \succeq 0$ for $r\ge 0$).
\end{prop}
\begin{proof}
  We show that the Schur power $K^{\circ r}=[k_{ij}^r]$, for $r > 0$ is positive definite. This claim follows after we realize that
\begin{equation*}
  K=[k_{ij}] = \left[\frac{1}{\max(i,j)}\right]\qquad 1\le i,j \le n,
\end{equation*}
which is well-known to be infinitely divisible~\citep[Ch.~5]{bhatia07}. We include a short proof below. Let $D=\Diag([i^{-1}]_{i=1}^n)$; also let $M:=[\min(i,j)]$. As  $K=DMD$, it suffices to establish infinite divisibility of $M$. We prove a more general statement. Let $f$ be a positive monotonic function, and for any set $C \subset \reals$, define $\ind_C(x)=1$ if $x \in C$ and $0$ otherwise. Then, $$f(m_{ij})=\min(f(i),f(j)) = \int_0^\infty \ind_{[0,f(i)]}(x)\ind_{[0,f(j)]}(x)dx,$$ which is nothing but an inner-product; thus $[f(m_{ij})]$ is a  Gram matrix and hence positive definite. Setting $f(t)=t^r, r > 0$, infinite divisibility of $M$ (and hence of $K$) follows.
\end{proof}

\subsection{Operator Norm bounds}
\label{sec.norm}
When using the kernels $\min(i,j)$ or $1/\max(i,j)$ in an application, one may need to bound their operator norms (e.g., for approximation or optimization). Let us first illustrate a bound on $\norm{K}$ that one may obtain naively. Since $K=DMD$, a naive bound is $\norm{K} \le \norm{D}^2\norm{M} = \norm{M}$ (as $\norm{D}=1$). Proposition~\ref{prop:lleig} bounds $\norm{M}$.

\begin{prop}
  \label{prop:lleig}
  Let $M=[m_{ij}]=[\min(i,j)]$; define $\theta_k := \frac{2k\pi}{2n+1}$. Then, $M=V\Lambda V^{-1}$ with  
  \begin{equation}
    \label{eq:3}
    \begin{split}
      \lambda_k &= (2+2\cos\theta_k)^{-1},\quad k=1,2,\ldots,n\\
      v_{jk}    &= \tfrac{2}{\sqrt{2n+1}}\sin (k-\half)\theta_j,\quad \theta_j \neq\pi,\qquad j,k=1,2,\ldots,n.
    \end{split}
  \end{equation}
  Consequently, $\norm{M} \approx 4n^2/\pi^2$.
\end{prop}
\begin{proof}
  The key is to observe that $M$ has the Cholesky factorization $M=LL^T$, where $L$ is the all ones lower-triangular matrix
  \begin{equation*}
    L = {\footnotesize\begin{pmatrix}
      1\\
      1&1\\
      \vdots&&\ddots\\
      1 &1 & \cdots & 1
    \end{pmatrix}.}
  \end{equation*}
  We wish to explicitly diagonalize $LL^T$, a task that becomes much easier if we consider the inverse $M^{-1}=L^{-T}L^{-1}$, as this is a perturbed Toeplitz-tridiagonal matrix
\begin{equation}
  \label{eq:1}
  L^{-T}L^{-1} =
  {\footnotesize\begin{pmatrix}
    2 & -1 &\\
    -1 & 2 & -1 &\\
    &\ddots&\ddots&\ddots&\\
    &&-1&2&-1\\
    &&&-1 & 1
  \end{pmatrix}.}
\end{equation}
Applying the derivation of~\citep[Thm.~3.2-(ix)]{yueh2008}\footnote{There seems to be a typo in the cited theorem; the cases (viii) and (ix) stated in that paper seem to be switched.}, we find that the eigenvalues of~\eqref{eq:1} are $\lambda_k^{-1}=2+2\cos\theta_k$ where $\theta_k = \frac{2k\pi}{2n+1}$. In fact $M^{-1}$ can also be diagonalized explicitly: its eigenvectors are given by (again by resorting to arguments of~\citep[Thm.~3.3]{yueh2008}):
\begin{equation}
  \label{eq:2}
  \tfrac{\sqrt{2n+1}}{2}v_{jk} := \sin (k-\half)\theta_j,\qquad j,k = 1,2,\ldots,n.
\end{equation}
With $V=[v_{jk}]$ and noting $V^{-1}=V^T$, we obtain  $V^T(LL^T)^{-1}V=\Diag(\lambda_k^{-1})$.

\noindent Hence, it immediately follows that $\norm{M} = \lambda_{\max}(LL^T)=(2+2\cos\theta_n)^{-1} \approx 4n^2/\pi^2$.
\end{proof}

The eigendecomposition of $M$ derived above is no surprise; $M=[\min(i,j)]$ is essentially the \emph{Brownian bridge covariance} function whose spectrum is well-studied~\citep{hofman06}. But it is worth noting that our derivation uses elementary linear algebra, compared with a more advanced Fourier-analytic derivation typically employed when studying eigenfunctions of kernels.

\begin{rmk}
  \label{rmk:one}
  Using $\norm{K}=\norm{DMD} \le \norm{D}^2\norm{M}\le \norm{M} \sim 4n^2/\pi^2$, we get a very pessimistic bound on $\norm{K}$. Remark~\ref{rmk:two} analyzes $\norm{K}$ directly, yielding a much better dependence on $n$. Finally, Proposition~\ref{prop:svals} actually provides a dimension independent bound on $\norm{K}$. %that $\norm{K}=\norm{P}^2 \le 4$. %However, if we use $M=\inv{D}K\inv{D}$ to bound $\norm{M}\le \norm{\inv{D}}^2\norm{K} \le 4n^2$, the bound is equally bad (its off by a factor of $n$).
\end{rmk}
\begin{rmk}
  \label{rmk:two}
  The bound from Remark~\ref{rmk:one} can be greatly improved rather easily. Indeed,
\begin{equation}
  \label{eq:8}
  \norm{K}=\norm{PP^T}=\lambda_{\max}(PP^T) \le \trace(PP^T) = \nlsum_i \tfrac1i = H_n,
\end{equation}
where $H_n$ denotes the $n$-th Harmonic number. However, even bound~\eqref{eq:8} is suboptimal as it depends on the dimension of the matrix $P$. %Prop.~\ref{prop:svals} actually establishes a dimension-independent bound.
\end{rmk}

\begin{prop}
  \label{prop:svals}
  Let $K=[1/\max(i,j)]$ for $1 \le i,j \le n$. Then, $\norm{K} \le 4$, while $\norm{K^{-1}} \le 4n^2$.
\end{prop}
\begin{proof}
  Since $K=PP^T$, we analyze $\norm{K}=\lambda_{\max}(PP^T)$, where $P$ is defined by~\eqref{eq:4}. 

  Let $Z=JP$. Then, $\norm{K}=\lambda_{\max}(JPP^TJ^{-1})=\lambda_{\max}(JP(JP)^T) = \norm{Z}^2$. A quick calculation show that $(I-Z)(I-Z^T) = \Diag[(n-j)/(n-j+1)]_{j=1}^n$. Thus,
  \begin{equation*}
    \norm{I-Z}^2 = \norm{(I-Z)(I-Z^T)} = \frac{n-1}{n}.
  \end{equation*}
  Thus, $\norm{I-Z} \le \sqrt{(n-1)/n}$, so by the triangle-inequality we obtain $\norm{Z} \le 1+\sqrt{(n-1)/n}$. But $Z=JP$ and the operator norm is unitarily invariant, hence $\norm{Z}=\norm{P} \le 1+\sqrt{(n-1)/n}$. Thus, $\norm{K}=\norm{P}^2 \le (1+\sqrt{1-1/n})^2 \le 4$.

  To compute $\norm{K^{-1}}$, consider 
  \begin{equation*}
    W := P^{-1}P^{-T} ={\footnotesize
    \begin{pmatrix}
      2(n-1)n+1 & -(n-1)^2 &\\
      -(n-1)^2  & 2(n-2)(n-1)+1 & -(n-2)^2 &\\
      &\hskip12pt\ddots&\\
      \\
      &&-4&5&-1\\
      &&&-1 & 1
    \end{pmatrix}},
  \end{equation*}
  which is a highly structured tridiagonal matrix. Thus, using Gershgorin's theorem, we obtain
  \begin{equation*}
    \lambda(W) \le |\lambda(W)-w_{ii}|+w_{ii} \le 2(n-i)^2+2(n-i)(n-i+1)+1,\qquad 1\le i \le n.
  \end{equation*}
  Thus, $\lambda_{\max}(W) \le 4n^2-6n+3 \le 4n^2$. Since $\norm{K^{-1}}=\norm{W}$, the proof is complete.
\end{proof}

\subsubsection*{Acknowledgments}
I would like to thank David Speyer, whose comment to an initial answer of mine on MathOverflow prodded me to think more carefully about the problem, which ultimately led to this paper.

\bibliographystyle{abbrvnat}
%\bibliography{markov}

\end{document}